\documentclass[12pt,psamsfonts]{amsart}

\usepackage{amssymb,eucal,epsfig}

\def\Dj{\hbox{D\kern-.73em\raise.30ex\hbox{-} \raise-.30ex\hbox{}}}
\def\dj{\hbox{d\kern-.33em\raise.80ex\hbox{-} \raise-.80ex\hbox{\kern-.40em}}}

\def\<{\langle}                     %added
\def\>{\rangle}                     %added
\def\N{\mathbb N}

\def\0{\mathbb 0}

\theoremstyle{plain}
\newtheorem{theorem}{Theorem}[section]

\newtheorem{lemma}{Lemma}[section]

\theoremstyle{definition}
\newtheorem{definition}{Definition}[section]

\theoremstyle{remark}
\newtheorem{remark}{Remark}[section]

\numberwithin{equation}{section}

\begin{document}
\setcounter{page}{1}

\vspace{25mm}

\baselineskip=0.25in
%\bibliography{catalan}

\title[\tiny Some$\;$considerations$\;$of$\;$matrix$\;$equations$\;$using$\;$the$\;$concept$\;$of$\;$reproductivity$\;$]
{Some considerations of matrix equations \\
using the concept of reproductivity}
\vspace{15mm}

\author[Branko Male\v sevi\' c]{Branko Male\v sevi\' c$^1$}
\address{$^1$Department of Applied Mathematics
\newline \indent Faculty of Electrical Engineering, University of Belgrade,
\newline \indent Serbia}
\email{malesevic@etf.rs}

\author[Biljana Radi\v ci\' c]{Biljana Radi\v ci\' c$^2$}
\address{$^2$ part-time job
\newline \indent Department of Mathematics, Physics and Descriptive Geometry
\newline \indent Faculty of Civil Engineering, University of Belgrade,
\newline \indent Serbia}
\email{radicic.biljana@yahoo.com}


\begin{abstract}
In this paper we analyse {\sc Cline}'s  matrix equation, generalized {\sc Penrose}'s matrix
system and a matrix system for $k$-commutative $\{1\}$-inverses. We determine reproductive and
non-reproductive general solutions of analysed matrix equation and analysed matrix systems.
\end{abstract}

\maketitle

\section{Introduction}

\noindent
In this paper we determine general and reproductive general solutions of analysed matrix equation
and analysed matrix systems. We are going to use the concept of reproductivity in order to prove
that certain formulas represent the general solutions of analysed matrix equation and analysed
matrix systems. The concept of reproductive equations was introduced by {\sc S.B. Pre\v si\'c}
\cite{Presic68} in 1968.

\noindent
Let $S$ be a given non-empty set and $J$ be a given unary relation of $S$.
Then an~equation $J(x)$ is {\em consistent} if there is at least one element $x_{0} \!\in\! S$,
so-called {\em the solution}, such~that $J(x_{0})$ is true. A formula $x=\phi(t)$, where $\phi: S
\rightarrow S$ is a given function, represents \mbox{\em the general solution} \cite{Bankovic11}
of the equation~$J(x)$ if and only if
$$
(\forall t) J(\phi(t)) \wedge (\forall x)( J(x) \Longrightarrow (\exists t) x=\phi(t) ).
$$
Let us cite the definition of reproductive equations according by {\sc S.B.~Pre\v si\'c}~\cite{Presic68}.

\begin{definition} \label{D1:D13}
\textit{The reproductive equations} are the equations of the following form:
$$
x=\varphi(x),
$$
where $x$ is a unknown, $S$ is a given set and  $\varphi:S \longrightarrow S$
is a given function which satisfies the following condition:
\begin{equation}
\label{UR}
\varphi\circ\varphi=\varphi.
\end{equation}
\end{definition}

\break

\noindent
The condition \eqref{UR} is called \textit{the condition of reproductivity} \cite{Presic68}.
The fundamental properties of reproductive equations are given by the following two statements
({\sc S.B. Pre\v si\'c} \cite{Presic68}) (see also \cite{Presic72}, \cite{Bozic75} and
\cite{Presic00}).

\begin{theorem} \label{T1:T11}
For any consistent equation $J(x)$ there is an equation of the form $x=\varphi(x)$, which is equivalent to $J(x)$ being in the same time reproductive as well.
\end{theorem}

\begin{theorem}\label{T1:T12}
If a certain equation $J(x)$ is equivalent to the reproductive one \mbox{$x\!\,=\!\,\varphi(x)$}, the
general solution is given by the formula $x=\varphi(y)$, for any value $y\in S$.
\end{theorem}

\noindent
Let us remark that a formula $x=\phi(t)$, where $\phi: S \rightarrow S$ is a given function, represents
\mbox{\em the reproductive general solution} \cite{Bankovic11} of the equation~$J(x)$
if and only if
$$
(\forall t) J(\phi(t)) \wedge (\forall t)( J(t) \Longrightarrow t = \phi(t) ).
$$

\noindent
Reproductivity of some equations of mathematical analysis was studied by {\sc J.D. Ke\v cki\' c}
in \cite{Keckic82}, \cite{Keckic83}. In \cite{KeckicPresic97} {\sc J.D. Ke\v cki\' c} and
{\sc S.B. Pre\v si\' c} considered the general applications of the concept of reproductivity.
The general applications of the concept of reproductivity in various mathematical
structures can also be found in \cite{Rudeanu78}, \cite{Bankovic79}, \cite{Bankovic95},
\cite{Bankovic96}, \cite{Keckic97} and~\cite{Presic00}.

\section{Main results}

\noindent
Let $m, n \in \N$ and $\mathbb{C}$ is the field of complex numbers. The set of all $m \times n$
matrices over $\mathbb{C}$ is denoted by $\mathbb{C}^{m\times n}$. By $\mathbb{C}_{a}^{m \times
n}$ we denote the set of all matrices from $\mathbb{C}^{m\times n}$ with a rank $a$. For
$A \!\in\!\mathbb{C}^{m\times n}$ the rank of $A$ is denoted by $rank(A).$ The unit matrix of
order $m$ is denoted by $I_{m}$ (if the dimension of unit matrix is known from the context, we
omit the index which indicates the dimension and we use designation $I$). Let $A \in
\mathbb{C}^{m \times n}$, then a solution of the matrix equation
\begin{equation*}
AXA=A
\end{equation*}
is called \textit{$\lbrace 1 \rbrace$-inverse} of $A$ and it is denoted by $ A^{(1)}.$
In the general case $\lbrace 1 \rbrace$-inverse of $A$ is not uniquely determined. The set of
all $\lbrace 1 \rbrace$-inverses of $A$ is denoted by $A\lbrace 1\rbrace$. It can be shown that
$A\lbrace 1\rbrace$ is not empty. $\lbrace 1 \rbrace$-inverse of $A$ is uniquely determined if
$A$ is regular. In that case $\lbrace 1\rbrace$-inverse $A^{(1)}$ corresponds to $A^{-1}$ i.e.
$A\lbrace 1\rbrace=\lbrace A^{-1}\rbrace$. There are also other types of inverses. More
informations about $\lbrace 1\rbrace$-inverse and other types of inverses can be found in
\cite{Ben-IsraelGreville03} and \cite{CampbellMeyer09}. For $A \!\in\!\mathbb{C}^{m\times m}$
the smallest non-negative integer $k$ such that $rank(A^{k})\!=\!rank(A^{k+1})$ is called the
index of $A$ and it is denoted~by~$Ind(A).$

\break

\medskip
\noindent
This section of paper is divided into \textit{three} parts. The \textit{first} part is devoted to
the matrix equation
\begin{equation}
\label{kj}
A^{m}XB^{n}=C,
\end{equation}where $A\!\in\!\mathbb{C}^{p \times p},$ $B\!\in\!\mathbb{C}^{q\times q},$
$C\!\in \!\mathbb{C}^{p \times q},$ $m\geq k=ind(A)$ and $n\geq l=ind(B)$.
In the \textit{second} part we consider the matrix system
\begin{equation}
\label{gps}
\hspace*{3 mm} (2.2.1.) \quad A^{m}X=B \quad \quad \wedge \quad \quad (2.2.2.) \quad XD^{n}=E,
\end{equation}
where $A\!\in \!\mathbb{C}^{p\times p}$, $B\!\in\! \mathbb{C}^{p \times q}$,
$D \!\in\!\mathbb{C}^{q\times q}$, $E \!\in\!\mathbb{C}^{p \times q}$, $m\geq Ind(A)$ and
$n\!\geq\!Ind(D)$. A solution of the matrix system
\begin{equation}
\label{kki}
AXA=A \quad \wedge \quad A^{k}X=XA^{k},
\end{equation} where $A\!\in \!\mathbb{C}^{p\times p}$ is a singular matrix and $k\in\N$, is
analysed in the \textit{third} part of this section.

\medskip
{\large \textbf{2.1.}} In this part we analyse the matrix equation \eqref{kj}. In the paper
\cite{Cline68}~\mbox{\sc R.E. Cline} was the first one who considered the matrix equation
\eqref{kj}. Using {\sc Penrose}'s condition for the consistence of the matrix equation $AXB=C$,
{\sc R.E. Cline} concluded that the matrix equation \eqref{kj} is consistent if and only if
\begin{equation}
\label{ukkj}
A^{m}(A^{m})^{(1)}C(B^{n})^{(1)}B^{n}=C.
\end{equation}
In the paper \cite{Cline68} it was shown that the matrix equation \eqref{kj} is consistent for
any $m>k$ and any $n>l$ if and only if the matrix equation $A^{k}XB^{l}=C$ is consistent. Based
on the results in the paper \cite{RadicicMalesevic} the condition of consistence \eqref{ukkj} for
the matrix equation \eqref{kj} can be also considered in a new form (see Theorem 2.1).

\begin{lemma}\label{L2:L21}
If the matrix equation \eqref{kj} is consistent, the equivalence
\begin{equation*}
A^{m}XB^{n}=C
\;\Longleftrightarrow\;
X=f(X)=X-(A^{m})^{(1)}(A^{m}XB^{n}-C)(B^{n})^{(1)}
\end{equation*}
is true.
\end{lemma}

\noindent
\begin{proof} $\Longrightarrow):$ Suppose that $A^{m}XB^{n}=C$. Then, the equality
$$(A^{m})^{(1)}A^{m}XB^{n}(B^{n})^{(1)}=(A^{m})^{(1)}C(B^{n})^{(1)}$$is also true and

$$
\begin{array}{rcl}
X&\!\!\!=\!\!\!&X-(A^{m})^{(1)}A^{m}XB^{n}(B^{n})^{(1)}+(A^{m})^{(1)}C(B^{n})^{(1)}

\medskip                                                              \\[1.0 ex]
&\!\!\!=\!\!\!&
X-(A^{m})^{(1)}(A^{m}XB^{n}-C)(B^{n})^{(1)}

\medskip                                                              \\[1.0 ex]
&\!\!\!=\!\!\!&
f(X)
\end{array}
$$

\smallskip
\noindent
$\Longleftarrow):$ Suppose that $X=f(X)=X-(A^{m})^{(1)}(A^{m}XB^{n}-C)(B^{n})^{(1)}$.
Then,
$$
\begin{array}{rcl}
A^{m}XB^{n}&\!\!\!=\!\!\!&A^{m}f(X)B^{n}

\medskip                                                              \\[1.0 ex]
&\!\!\!=\!\!\!&
A^{m}\big{(}X-(A^{m})^{(1)}(A^{m}XB^{n}-C)(B^{n})^{(1)}\big{)}B^{n}

\medskip                                                              \\[1.0 ex]
&\!\!\!=\!\!\!&
A^{m}XB^{n}-
\mathop{\underbrace{A^{m}(A^{m})^{(1)}A^{m}}}\limits_{(=A^{m})}X
\mathop{\underbrace{B^{n}(B^{n})^{(1)}B^{n}}}\limits_{(=B^{n})}
+\mathop{\underbrace{A^{m}(A^{m})^{(1)}C(B^{n})^{(1)}B^{n}}}
\limits_{(\mathop{=} \limits_{\eqref{ukkj}}\!C)}

\medskip                                                              \\[1.0 ex]
&\!\!\!=\!\!\!&
A^{m}XB^{n}-A^{m}XB^{n}+C

\medskip                                                              \\[1.0 ex]
&\!\!\!=\!\!\!&
C.
\end{array}
$$

\vspace*{-4.0 mm}

\end{proof}

\begin{remark} \label{R2:R21}
It is easy to show that $f^{2}(Y)=f(Y)$ i.e. the function $f$ satisfies the condition of
reproductivity. Therefore, if the matrix equation \eqref{kj} is consistent, it is equivalent to
the reproductive matrix equation $X\!=\!f(X)$.
\end{remark}

\noindent
Based on the previous remark and Theorem \ref{T1:T12} we conclude that the following theorem
is true.

\begin{theorem}\label{T2:T21}
If the matrix equation \eqref{kj} is consistent, the general solution of the matrix equation
\eqref{kj} is given by the formula
\begin{equation*}
X=f(Y)=(A^{m})^{(1)}C(B^{n})^{(1)}+Y-(A^{m})^{(1)}A^{m}YB^{n}(B^{n})^{(1)},
\end{equation*}
where $Y$ is an arbitrary matrix corresponding dimensions.
\end{theorem}

\noindent
The following theorem is an extension of the previous theorem.

\begin{theorem}\label{T2:T22}
If $X_{0}$ is a particular solution of the matrix equation \eqref{kj}, the general solution of
the matrix equation \eqref{kj} is given by the formula
\begin{equation}
\label{oorkj}
X=g(Y)=X_{0}+Y-(A^{m})^{(1)}A^{m}YB^{n}(B^{n})^{(1)},
\end{equation}
where $Y$ is an arbitrary matrix corresponding dimensions.
\end{theorem}

\break

\noindent
\begin{proof} It is easy to see that the solution of the matrix equation \eqref{kj} is
given by \eqref{oorkj}. On the contrary, let $X$ be any solution of the matrix equation
\eqref{kj}, then
\begin{eqnarray*}
X&\!\!\!=\!\!\! &X-(A^{m})^{(1)}C(B^{n})^{(1)}+(A^{m})^{(1)}C(B^{n})^{(1)}\\
 &\!\!\!=\!\!\! &X-(A^{m})^{(1)}A^{m}XB^{n}(B^{n})^{(1)}+(A^{m})^{(1)}A^{m}X_{0}B^{n}
 (B^{n})^{(1)}\\
 &\!\!\!=\!\!\! &X-(A^{m})^{(1)}A^{m}(X-X_{0})B^{n}(B^{n})^{(1)}\\
 &\!\!\!=\!\!\! &X_{0}+(X-X_{0})-(A^{m})^{(1)}A^{m}(X-X_{0})B^{n}(B^{n})^{(1)}\\
 &\!\!\!=\!\!\! &X_{0}+Y-(A^{m})^{(1)}A^{m}YB^{n}(B^{n})^{(1)}\\
 &\!\!\!=\!\!\! &g(Y),
\end{eqnarray*}
where  $Y=X-X_{0}$. From this we see that every solution $X$ of the matrix equation \eqref{kj}
can be represented in the form \eqref{oorkj}.
\end{proof}

\noindent
\begin{remark} \label{R2:R22}
From $g^{2}(Y)=g(Y)+\mbox{\big(}X_{0}-(A^{m})^{(1)}C(B^{n})^{(1)}
\mbox{\big)}$ we conclude that the function $g$ is reproductive if and only if
$X_{0}=(A^{m})^{(1)}C(B^{n})^{(1)}.$
\end{remark}

\noindent
\begin{remark} \label{R2:R23}
Theorem \ref{T2:T22} is an extension, as we mentioned, of Theorem \ref{T2:T21} because there is a
matrix equation \eqref{kj} and a particular solution $X_{0}$ such that
$X_{0} \neq (A^{m})^{(1)}C(B^{n})^{(1)}$ for any choice of $\lbrace 1\rbrace$-inverses
$(A^{m})^{(1)}$ and $(B^{n})^{(1)}$ similar to the corresponding example from
\cite{MalesevicRadicic}.
\end{remark}

\bigskip
{\large \textbf{2.2.}} In this part we analyse the matrix system \eqref{gps}, as a special extension of
{\sc Penrose}'s matrix system \cite{Penrose55}:
\begin{equation*}
\quad AX=B \quad \quad \wedge \quad \quad XD=E,
\end{equation*}
using the concept of reproductivity.

\medskip
\noindent
Based on the result from \cite{Penrose55} we conclude that one common solution of the matrix
system \eqref{gps} is given by
$$
X_{1}=(A^{m})^{(1)}B+E(D^{n})^{(1)}-(A^{m})^{(1)}(A^{m})AE(D^{n})^{(1)}.
$$
The results which follow are extensions of the results from \cite{Ben-IsraelGreville03}
(pp. 54-55) and \cite{RadicicMalesevic}.

\noindent
\begin{lemma}\label{L2:L22}
The matrix equations \mbox{\rm (2.2.1.)} and \mbox{\rm (2.2.2.)} have a common solution if and
only if each equation separately has a solution and $$A^{m}E=BD^{n}.$$
\end{lemma}

\begin{proof}
The proof is similar to the proof in \cite{Ben-IsraelGreville03}.
\end{proof}

\noindent
\begin{lemma}\label{L2:L23}
If the matrix system \eqref{gps} is consistent, the equivalence
\begin{equation*}
\begin{array}{l}
(A^{m}X=B \;\; \wedge \;\; XD^{n}=E ) \;\; \Longleftrightarrow \;\;  \\[2.0 ex]
X=f(X)=X_{1}+(I-(A^{m})^{(1)}A^{m})X(I-D^{n}(D^{n})^{(1)})
\end{array}
\end{equation*}
 is true.
\end{lemma}

\noindent
\begin{proof}
The proof is similar to the proof in \cite{RadicicMalesevic}.
\end{proof}

\noindent
\begin{remark}\label{R2:24}
It is easy to show that $ f^{2}(Y)=f(Y)$ i.e. the function $f$ satisfies the condition of
reproductivity. Therefore, if the matrix system \eqref{gps} is consistent, it is equivalent to
the reproductive matrix equation $X\!=\!f(X)$.
\end{remark}

\noindent
Based on the previous remark and Theorem \ref{T1:T12} we conclude that the following theorem is
true.

\noindent
\begin{theorem}\label{T2:T23}
If the matrix system \eqref{gps} is consistent, the general solution of the matrix system
\mbox{\rm (\ref{gps})} is given by the formula
\begin{equation*}
X=f(Y)=X_{1}+(I-(A^{m})^{(1)}A^{m})Y(I-D^{n}(D^{n})^{(1)}),
\end{equation*}
where $Y$ is an arbitrary matrix corresponding dimensions.
\end{theorem}

\noindent
The following theorem is an extension of the previous theorem.

\noindent
\begin{theorem}\label{T2:T24}
If $X_{0}$ is a particular solution of the matrix system \eqref{gps}, the general solution of the
matrix system \eqref{gps}) is given by the formula
\begin{equation*}
X=g(Y)=X_{0}+(I-(A^{m})^{(1)}A^{m})Y(I-D^{n}(D^{n})^{(1)}),
\end{equation*}
where $Y$ is an arbitrary matrix corresponding dimensions.
\end{theorem}

\noindent
\begin{proof} The proof is similar to the proof in \cite{RadicicMalesevic}.\end{proof}

\noindent
\begin{remark} \label{R2:R25}
From $g^{2}(Y)=g(Y)+(X_{0}-X_{1})$ we conclude that the function $g$ is reproductive if and only if $X_{0}=X_{1}.$
\end{remark}

\break

\bigskip
{\large \textbf{2.3.}} In this part we analyse the matrix system \eqref{kki}. The second equation of the matrix system \eqref{kki}
determines {\it $\{5^{k}\}$-inverse} of $A$. A solution of the matrix system \eqref{kki} is $\{1,5^{k}\}$-inverse which
is called {\it $k$-commutative  $\{1\}$-inverse} and is denoted~by~$\bar{A}$. $k$-commutative $\{1\}$-inverses were considered
in \cite{Ruski_Gruv_84_85}, \cite{Keckic85} and \cite{Keckic97}. It is easy to check that one solution of the matrix system
\eqref{kki} is given by $\hat{X} \!=\! \bar{A}A\bar{A}$. In \cite{Keckic85} {\sc J.D. Ke\v cki\' c} gave the condition for
the consistency of the matrix system \eqref{kki}. We are going to represent the formula of the general reproductive solution
for the consistent matrix system \eqref{kki} using the concept of reproductive equations. We need the following four lemmas.

\noindent
\begin{lemma}\label{L2:L24}
$A^{k} \bar{A}^{k} = \bar{A}^{k} A^{k}$.
\end{lemma}

\noindent
\begin{proof}
$A^{k}\bar{A}^{k}\!=\!
\mathop{\underbrace{A^{k}\bar{A}}}\limits_{(=\bar{A}A^{k})}\!\!\bar{A}^{k-1}$
$\!=\!\bar{A}A^{k}\bar{A}^{k-1}\!=
\!\bar{A}\!\!\mathop{\underbrace{A^{k}\bar{A}}}\limits_{(=\bar{A}A^{k})}\!\!\bar{A}^{k-2}\!$
$=\bar{A}^{2}A^{k}\bar{A}^{k-2}$
$=\ldots=\!\bar{A}^{k} A^{k}.$
\end{proof}

\noindent
\begin{lemma}\label{L2:L25}
$A^{k} \bar{A}^{k} A^{k} = A^{k}$.
\end{lemma}

\noindent
\begin{proof} Let us note $A^{k} = A^{k-1} \!\!\!\!\!\!\mathop{\underbrace{A}}\limits_{(=A\bar{A}A)}\!\!\!\!
= A^{k-1} A \bar{A} A = \mathop{\underbrace{A^{k}\bar{A}}}\limits_{(=\bar{A}A^{k})} \!\!\! A = \bar{A}A^{k+1}$.
Therefore,

\smallskip

\qquad
$\!A^{k} = \bar{A} (A^{k}) \, A = \bar{A} (\bar{A} A^{k+1}) \,  A = \bar{A}^2 A^{k} A^{2} = \ldots = \bar{A}^{k} A^{k} A^{k}
\mathop{=}\limits_{\mbox{\tiny L. \ref{L2:L24}.}}
A^{k} \bar{A}^{k} A^{k}$.
\end{proof}

\noindent
\begin{lemma} \label{L2:L26}
For any particular solution $X_{0}$ of the matrix system \eqref{kki} the equalities
$$
X_{0} A^{k} \bar{A}^{k} = A^{k} \bar{A}^{k+1}
\quad  \textit{\mbox{and}} \quad
\bar{A}^{k} A^{k} X_{0} = A^{k} \bar{A}^{k+1}
$$ are true.
\end{lemma}

\noindent
\begin{proof} We are going to prove the first equality.

\bigskip
\noindent
$$
\begin{array}{rcl}
X_{0}\!\!\!\!\!\mathop{\underbrace{A^{k}}}\limits_{(\mathop{=}\limits_{\mbox{\tiny L.\ref{L2:L25}.}} \! A^{k}\bar{A}^{k} A^{k})}\!\!\!\!\!
\bar{A}^{k}
&\!\!\!=\!\!\!&\mathop{\underbrace{X_{0}A^{k}}}\limits_{(=A^{k}X_{0})}
\!\mathop{\underbrace{\bar{A}^{k}A^{k}}}\limits_{(\mathop{=}\limits_{\mbox{\tiny L.\ref{L2:L24}.}}
\!\!A^{k}\bar{A}^{k})}\bar{A}^{k}
=
A^{k}X_{0}A^{k}\bar{A}^{2k}

\smallskip                                                             \\[0.5 ex]
&\!\!\!=\!\!\!&
A^{k-1}\mathop{\underbrace{AX_{0}A}}\limits_{(=A)}A^{k-1}\bar{A}^{2k}
=
A^{k-1}AA^{k-1}\bar{A}^{2k}
=
A^{2k-1}\bar{A}^{2k},
\end{array}
$$

$$
\begin{array}{rcl}
\mathop{\underbrace{ A^{k}}}\limits_{(\mathop{=}\limits_{\mbox{\tiny L.\ref{L2:L25}.}} \! A^{k}\bar{A}^{k} A^{k})}\!\!\!\!\!\bar{A}^{k+1}
&\!\!\!=\!\!\!&
A^{k}\!\!\!\!\!\mathop{\underbrace{\bar{A}^{k}A^{k}}}
\limits_{(\mathop{=}\limits_{\mbox{\tiny L.\ref{L2:L24}.}} \!\! A^{k}\bar{A}^{k})}\!\!\!\!\!\bar{A}^{k+1}
=
A^{2k}\bar{A}^{2k+1}
=
A^{k}\!\!\!\mathop{\underbrace{A^{k}\bar{A}}}\limits_{(=\bar{A}A^{k})}\!\!\!\bar{A}^{2k}
=
A^{k}\bar{A}A^{k}\bar{A}^{2k}

\smallskip                                                             \\[0.5 ex]
&\!\!\!=\!\!\!&
A^{k-1}\mathop{\underbrace{A\bar{A}A}}\limits_{(=A)}A^{k-1}\bar{A}^{2k}
=
A^{k-1}AA^{k-1}\bar{A}^{2k}
=
A^{2k-1}\bar{A}^{2k}.
\end{array}
$$

\medskip
\noindent
Hence, $X_{0} A^{k} \bar{A}^{k} = A^{k} \bar{A}^{k+1}.$ The second equality is proved similarly.
\end{proof}

\break

\noindent
\begin{lemma}\label{L2:L27}
Let $\hat{X}=\bar{A}A\bar{A}$. If the matrix system \eqref{kki} is consistent, the
equivalence
\begin{equation*}
\begin{array}{l}
AXA=A \hspace*{1 mm} \wedge \hspace*{1 mm} A^{k}X=XA^{k} \quad \Longleftrightarrow \quad                           \\ [2.0 ex]
X\!=\!f(X)\!=\!\hat{X}\!+\!X\!-\!(I\!-\!\bar{A}A)XA^{k}\bar{A}^{k}\!-\!\bar{A}^{k}A^{k}X(I\!-\!A\bar{A})\!-\!\bar{A}AXA\bar{A}
\end{array}
\end{equation*}
is  true.
\end{lemma}

\vspace*{-4.0 mm}

\begin{proof} $\Longrightarrow ):$ Suppose that $AXA=A \hspace*{1.5 mm} \wedge \hspace*{1.5 mm}
A^{k}X=XA^{k}$. Then,
$$\bar{A}AXA^{k}\bar{A}^{k}
=\bar{A}\mathop{\underbrace{AXA}}\limits_{(=A)}A^{k-1}\bar{A}^{k}
=\mathop{\underbrace{\bar{A}A^{k}}}\limits_{(=A^{k}\bar{A})}\!\!\bar{A}^{k}=
A^{k}\bar{A}\bar{A}^{k}=A^{k}\bar{A}^{k+1}.$$
Bearing in mind that $X A^{k}\bar{A}^{k} = A^{k}\bar{A}^{k+1}$ (Lemma \ref{L2:L26}), we conclude that
$$
XA^{k}\bar{A}^{k} = \bar{A}AXA^{k}\bar{A}^{k}.
$$

\smallskip
\noindent
In a similar way we get that
$
\bar{A}^{k}A^{k}X=\bar{A}^{k}A^{k}XA\bar{A} .
$

\medskip
\noindent
Therefore,

\medskip
\noindent
$X=\bar{A}A\bar{A}+X-XA^{k}\bar{A}^{k}+\bar{A}AXA^{k}\bar{A}^{k}-\bar{A}^{k}A^{k}X
+\bar{A}^{k}A^{k}XA\bar{A}-\bar{A}\!\!\!\!\mathop{\underbrace{A}}\limits_{(=AXA)}\!\!\!\!\bar{A}$

$=\bar{A}A\bar{A}+X-XA^{k}\bar{A}^{k}+\bar{A}AXA^{k}\bar{A}^{k}-\bar{A}^{k}A^{k}X
+\bar{A}^{k}A^{k}XA\bar{A}-\bar{A}AXA\bar{A}$

\medskip
$=\bar{A}A\bar{A}+X-(I-\bar{A}A)XA^{k}\bar{A}^{k}-\bar{A}^{k}A^{k}X(I-A\bar{A})
-\bar{A}AXA\bar{A}$

\medskip
$=f(X).$

\medskip
\noindent
$\Longleftarrow):$ Suppose that
\mbox{$X=f(X)=\bar{A}A\bar{A}+X-(I-\bar{A}A)XA^{k}\bar{A}^{k}-\bar{A}^{k}A^{k}X(I-A\bar{A})$}
$-\bar{A}AXA\bar{A}$.
Then,
$$
\begin{array}{rcl}
AXA &\!\!\!=\!\!\!& Af(X)A

\medskip                                                             \\[2.0 ex]
&\!\!\!=\!\!\!&
A{\big (}\bar{A}A\bar{A}+X-(I-\bar{A}A)XA^{k}\bar{A}^{k}-\bar{A}^{k}A^{k}X(I-A\bar{A})
-\bar{A}AXA\bar{A}{\big )}A

\medskip                                                             \\[2.5 ex]
&\!\!\!=\!\!\!&
\mathop{\underbrace{A\bar{A}A}}\limits_{
(=A)}\bar{A}A\!+\!AXA\!-\!\mathop{\underbrace{A(I\!-\!\bar{A}A)}}\limits_{
(=0)}XA^{k}\bar{A}^{k}A\!-\!A\bar{A}^{k}A^{k}X\mathop{\underbrace{(I\!-\!A\bar{A})A}}\limits_{
(=0)}\!-\!\mathop{\underbrace{A\bar{A}A}}\limits_{
(=A)}X\mathop{\underbrace{A\bar{A}A}}\limits_{(=A)}

\medskip                                                             \\[2.5 ex]
&\!\!\!=\!\!\!&
\mathop{\underbrace{A\bar{A}A}}\limits_{(=A)}+AXA-AXA

\medskip                                                             \\[2.5 ex]
&\!\!\!=\!\!\!&A,
\end{array}
$$
$$
\begin{array}{rcl}
A^{k}X &\!\!\!=\!\!\!& A^{k}f(X)

\medskip                                                             \\[1.5 ex]
&\!\!\!=\!\!\!&
A^{k}{\big (}\bar{A}A\bar{A}+X-(I-\bar{A}A)XA^{k}\bar{A}^{k}-\bar{A}^{k}A^{k}X(I-A\bar{A})
-\bar{A}AXA\bar{A}{\big )}

\medskip                                                             \\[1.5 ex]
&\!\!\!=\!\!\!&
A^{k}\bar{A}A\bar{A}+A^{k}X\!-\!A^{k}(I\!-\!\bar{A}A)XA^{k}\bar{A}^{k}
\!-\!\mathop{\underbrace{A^{k}\bar{A}^{k}A^{k}}}\limits_{(\mathop{=}\limits_{\mbox{\tiny L.\ref{L2:L25}.}} \! A^{k})}
X(I\!-\!A\bar{A})\!-\!A^{k}\bar{A}\mathop{\underbrace{AXA}}\limits_{(=A)}\bar{A}

\medskip                                                             \\[1.5 ex]
&\!\!\!=\!\!\!&A^{k}\bar{A}A\bar{A}+
A^{k}X\!-\!A^{k}XA^{k}\bar{A}^{k}+A^{k}\bar{A}AXA^{k}\bar{A}^{k}
\!-\!A^{k}X+A^{k}XA\bar{A}\!-\!A^{k}\bar{A}A\bar{A}

\medskip                                                             \\[1.5 ex]
&\!\!\!=\!\!\!&
-A^{k}\!\!\!\!\mathop{\underbrace{XA^{k}\bar{A}^{k}}}
\limits_{(\mathop{=} \limits_{\mbox{\tiny L.\ref{L2:L26}.}} \!\! A^{k}\bar{A}^{k+1})}\!\!+ A^{k}\bar{A}
\mathop{\underbrace{AXA}}\limits_{(=A)}A^{k-1}\bar{A}^{k}
+A^{k-1}\mathop{\underbrace{AXA}}\limits_{(=A)}\bar{A}

\medskip                                                             \\[1.5 ex]
&\!\!\!=\!\!\!&
-A^{k}A^{k}\bar{A}^{k+1}+A^{k}\!\!\!\!\mathop{\underbrace{\bar{A}A^{k}}}\limits_{(=A^{k}\bar{A})}
\!\!\bar{A}^{k}+A^{k}\bar{A}

\medskip                                                             \\[1.5 ex]
&\!\!\!=\!\!\!&
-A^{2k}\bar{A}^{k+1}+A^{k}A^{k}\bar{A}\bar{A}^{k}+A^{k}\bar{A}

\medskip                                                             \\[1.5 ex]
&\!\!\!=\!\!\!&
-A^{2k}\bar{A}^{k+1}+A^{2k}\bar{A}^{k+1}+A^{k}\bar{A}

\medskip                                                             \\[1.5 ex]
&\!\!\!=\!\!\!&
A^{k}\bar{A},
\end{array}
$$

$$
\begin{array}{rcl}
XA^{k}&\!\!\!=\!\!\!&f(X)A^{k}

\medskip                                                             \\[1.5 ex]
&\!\!\!=\!\!\!&
{\big (}\bar{A}A\bar{A}+X-(I-\bar{A}A)XA^{k}\bar{A}^{k}-\bar{A}^{k}A^{k}X(I-A\bar{A})
-\bar{A}AXA\bar{A}{\big )}A^{k}

\medskip                                                             \\[1.5 ex]
&\!\!\!=\!\!\!&
\bar{A}A\bar{A}A^{k}\!+\!XA^{k}\!-\!(I\!-\!
\bar{A}A)X\mathop{\underbrace{A^{k}\bar{A}^{k}A^{k}}}\limits_{(\mathop{=}\limits_{\mbox{\tiny L.\ref{L2:L25}.}} \! A^{k})}
\!-\!\bar{A}^{k}A^{k}X(I\!-\!A\bar{A})A^{k}\!-\!
\bar{A}\mathop{\underbrace{AXA}}\limits_{(=A)}\bar{A}A^{k}

\medskip                                                             \\[1.5 ex]
&\!\!\!=\!\!\!&
\bar{A}A\bar{A}A^{k}\!+\!XA^{k}\!-\!XA^{k}\!+\!\bar{A}AXA^{k}
\!-\!\bar{A}^{k}A^{k}XA^{k}\!+\!\bar{A}^{k}A^{k}XA\bar{A}A^{k}\!-\!\bar{A}A\bar{A}A^{k}

\medskip                                                             \\[1.5 ex]
&\!\!\!=\!\!\!&
\bar{A}\mathop{\underbrace{AXA}}\limits_{(=A)}A^{k-1}-\!\!\!\!
\mathop{\underbrace{\bar{A}^{k}A^{k}X}}\limits_{(\mathop{=}\limits_{\mbox{\tiny L.\ref{L2:L26}.}}
\!\! A^{k}\bar{A}^{k+1})}\!\!\!\!A^{k}+\bar{A}^{k}A^{k}XA\bar{A}A^{k}

\medskip                                                             \\[1.5 ex]
&\!\!\!=\!\!\!&
\bar{A}A^{k}-A^{k}\bar{A}^{k+1}A^{k}+\bar{A}^{k}A^{k}XA\bar{A}A^{k}

\medskip                                                             \\[1.5 ex]
&\!\!\!=\!\!\!&
\bar{A}A^{k}-A^{k}\bar{A}^{k+1}A^{k}+
\bar{A}^{k}A^{k-1}\mathop{\underbrace{AXA}}\limits_{(=A)}\bar{A}A^{k}

\medskip                                                             \\[1.5 ex]
&\!\!\!=\!\!\!&
\bar{A}A^{k}-A^{k}\bar{A}^{k+1}A^{k}+\!\!
\mathop{\underbrace{\bar{A}^{k}A^{k}}}
\limits_{(\mathop{=}\limits_{\mbox{\tiny L.\ref{L2:L24}.}}\!\!\!
A^{k}\bar{A}^{k})}\!\!\!\!\bar{A}A^{k}

\medskip                                                             \\[1.5 ex]
&\!\!\!=\!\!\!&
\bar{A}A^{k}-A^{k}\bar{A}^{k+1}A^{k}+A^{k}\bar{A}^{k+1}A^{k}

\medskip                                                             \\[1.5 ex]
&\!\!\!=\!\!\!&
\bar{A}A^{k}.
\end{array}
$$

\noindent
From $A^{k}\bar{A}=\bar{A}A^{k}$ we see that $A^{k}X=XA^{k}.$
\end{proof}

\break

\noindent
\begin{remark} \label{R2:R26}
It is easy to show that $f^{2}(Y)=f(Y)$ i.e. the function $f$ satisfies the condition of
reproductivity. Therefore, if the matrix system \eqref{kki} is consistent, it is equivalent to
the reproductive matrix equation $X \!=\!f(X)$.
\end{remark}

\noindent
Based on the previous remark and Theorem \ref{T1:T12} we conclude that the following theorem is
true.

\noindent
\begin{theorem} \label{T2:T25}
If the matrix system \eqref{kki} is consistent, the general solution of the matrix system
\eqref{kki} is given by the formula
\begin{equation*}
X = f(Y) =\!\bar{A}A\bar{A} + Y - (I\!-\!\bar{A}A)Y\!A^{k}\!\bar{A}^{k} -
\bar{A}^{k}\!A^{k}Y(I\!-\!A\bar{A}) - \bar{A}AY\!A\bar{A},
\end{equation*}
where $Y$ is an arbitrary matrix corresponding dimensions.
\end{theorem}

\medskip
\noindent
In \cite{Keckic85} {\sc J.D. Ke\v cki\' c} also proved this theorem, but his proof is different
from the previously exposed proof.

\noindent
The following theorem is an extension of the previous theorem.

\noindent
\begin{theorem} \label{T2:T26}
If $X_{0}$ is a particular solution of the matrix equation \eqref{kki}, the general solution of
the matrix equation \eqref{kki} is given by the formula
\begin{equation}
\label{porkki}
X = g(Y) =\!X_{0} + Y - (I\!-\!\bar{A}A)Y\!A^{k}\!\bar{A}^{k} -
\bar{A}^{k}\!A^{k}Y(I\!-\!A\bar{A}) - \bar{A}AY\!A\bar{A}.
\end{equation}
where $Y$ is an arbitrary matrix corresponding dimensions.
\end{theorem}

\noindent
\begin{proof}
It is easy to see that the solution of the matrix system \eqref{kki} is given by
\eqref{porkki}. On the contrary, let $X$ be any solution of the matrix equation \eqref{kki},
then
\begin{eqnarray*}
X\,&\!\!\!\!=\!\!\!\!&X-\!\!\mathop{\underbrace{A^{k}\bar{A}^{k+1}}}
\limits_{(\mathop{=}\limits_{\mbox{\tiny L.\ref{L2:L26}.}} \!\! XA^{k}\bar{A}^{k})}\!\!+\!\!
\mathop{\underbrace{A^{k}\bar{A}^{k+1}}}\limits_{(\!\mathop{=}\limits_{\mbox{\tiny L.\ref{L2:L26}.}} \!\!
X_{0}A^{k}\bar{A}^{k})}\!\!\! + \,\bar{A}A^{k+1}\bar{A}^{k+1}-\bar{A}A^{k+1}\bar{A}^{k+1}-\!\!\!
\mathop{\underbrace{A^{k}\bar{A}^{k+1}}}
\limits_{(\mathop{=}\limits_{\mbox{\tiny L.\ref{L2:L26}.}} \!\! \bar{A}^{k}A^{k}X)}\\[1.5 ex]
&\!\!\!\!\!\!\!\!&+\!\!\!\mathop{\underbrace{A^{k}\bar{A}^{k+1}}}
\limits_{(\mathop{=}\limits_{\mbox{\tiny L.\ref{L2:L26}.}} \!\! \bar{A}^{k}A^{k}X)}\!\!\!A\bar{A} \;+\!\!
\mathop{\underbrace{A^{k}\bar{A}^{k+1}}}\limits_{(\mathop{=}\limits_{\mbox{\tiny L.\ref{L2:L26}.}} \!\!
\bar{A}^{k}A^{k}X_{0})} \!\!-\!\! \mathop{\underbrace{A^{k}\bar{A}^{k+1}}}
\limits_{(\mathop{=}\limits_{\mbox{\tiny L.\ref{L2:L26}.}} \!\! \bar{A}^{k}A^{k}X_{0})}\!\!\!\!
A\bar{A} \,-\, \bar{A}\!\!\!\! \mathop{\underbrace{A}}\limits_{(=AXA)}\!\!\!\!\bar{A} \,+\, \bar{A}\!\!\!
\!\!\!\mathop{\underbrace{A}}\limits_{(=AX_{0}A)}\!\!\!\!\!\!\bar{A}\\[1.5 ex]
&\!\!\!\!=\!\!\!\!&X-XA^{k}\bar{A}^{k}+X_{0}A^{k}\bar{A}^{k}+\bar{A}A
\!\!\! \mathop{\underbrace{A^{k}\bar{A}^{k+1}}}\limits_{(\mathop{=}\limits_{\mbox{\tiny L.\ref{L2:L26}.}} \!\!
XA^{k}\bar{A}^{k})}\!-\,\bar{A}A\!\!\!
\mathop{\underbrace{A^{k}\bar{A}^{k+1}}}\limits_{(\mathop{=}\limits_{\mbox{\tiny L.\ref{L2:L26}.}} \!\!
  X_{0}A^{k}\bar{A}^{k})}\!\!-\,\bar{A}^{k}A^{k}X\\[1.0 ex]
  & \!\!\!\!\!\!\!\!&+\,\bar{A}^{k}A^{k}XA\bar{A}
  +\bar{A}^{k}A^{k}X_{0}-\bar{A}^{k}A^{k}X_{0}A\bar{A}-\bar{A}AXA\bar{A}
  +\bar{A}AX_{0}A\bar{A}
\end{eqnarray*}
\begin{eqnarray*}
  &\!\!\!\!=\!\!\!\!&X_{0}\!+\!(X\!-\!X_{0})\!-\!(X\!-\!X_{0})A^{k}\bar{A}^{k}\!+\!\bar{A}A
  XA^{k}\bar{A}^{k}
 \! -\!\bar{A}A X_{0}A^{k}\bar{A}^{k}\!-\!\bar{A}^{k}A^{k}(X\!-\!X_{0}) \\[0.5 ex]
  & \!\!\!\!\!\!\!\!&+\,
  \bar{A}^{k}A^{k}(X-X_{0})A\bar{A}-\bar{A}A(X-X_{0})A\bar{A}\\[0.5 ex]
  &\!\!\!\!=\!\!\!\!&X_{0}+(X-X_{0})-(X-X_{0})A^{k}\bar{A}^{k}+
  \bar{A}A( X-X_{0})A^{k}\bar{A}^{k}-\bar{A}^{k}A^{k}(X-X_{0})\\[0.5 ex]
  & \!\!\!\!\!\!\!\!&+\,\bar{A}^{k}A^{k}(X-X_{0})A\bar{A}-\bar{A}A(X-X_{0})A\bar{A}\\[0.5 ex]
  &\!\!\!\!=\!\!\!\!&X_{0}+(X-X_{0})-(I-\bar{A}A)(X-X_{0})A^{k}\bar{A}^{k}
  -\bar{A}^{k}A^{k}(X-X_{0})(I-A\bar{A})\\[0.5 ex]
  & \!\!\!\!\!\!\!\!&-\,\bar{A}A(X-X_{0})A\bar{A}\\[0.5 ex]
  &\!\!\!\!=\!\!\!\!& X_{0} + Y - (I\!-\!\bar{A}A)Y\!A^{k}\!\bar{A}^{k} -
  \bar{A}^{k}\!A^{k}Y(I\!-\!A\bar{A}) - \bar{A}AY\!A\bar{A}\\[0.5 ex]
  &\!\!\!\!=\!\!\!\!&g(Y),
\end{eqnarray*}
where  $Y=X-X_{0}$. From this we see that every solution $ X $ of the matrix system (\ref{gps})
can be represented in the form \eqref{porkki}.
\end{proof}

\noindent
\begin{remark} \label{R2:R27}
From $g^{2}(Y)=g(Y)+(X_{0}-\hat{X})$ we conclude that the function $g$
is reproductive if and only if $X_{0}=\hat{X}.$
\end{remark}

\noindent
\begin{remark} \label{R2:R28}
The preceding result is an extension of the consideration which is given in
\cite{RadicicMalesevic} (see Application 2.2)
\end{remark}

\section{Conclusion}

\noindent
We want to emphasize that there are also other matrix equations and matrix systems whose
solutions can be analysed in the same way as we done with \eqref{kj}, \eqref{gps} and
\eqref{kki}.

\vspace{5mm}

{\bf Acknowledgment.} The research is partially supported by the Ministry of Science and Education,
Serbia, Grant No. 174032.

\bigskip


\begin{thebibliography}{99}

\bibitem{Penrose55} R. Penrose, {\it A generalized inverses for matrices}, Math. Proc. Cambridge
Philos. Soc. {\bf 51} (1955), 406--413.

\bibitem{Presic68} S.B. Pre\v si\' c, {\it Une classe d'\' equations  matricielles et l'\'
equation fonctionnelle $f^{2}\!=\!f$}, Publications de  l'institut  mathematique, Nouvelle serie,
tome 8 {\bf 22}, Beograd 1968, 143--148. (http:/$\!$/publications.mi.sanu.ac.rs/)

\bibitem{Cline68} R.E. Cline, {\it Inverses of rank invariant powers of a matrix}, SIAM J.
Numer. Anal., {\bf 5} (1968), 182--197.

\bibitem{Presic71} S.B. Pre\v si\' c, {\it Une m\'ethode de r\'esolution des  \'eequations
dont toutes les solutions appartiennent  \'a un ensemble fini donn\'e}, C. R. Acad. Sci. Paris
Ser. A {\bf 272} (1971) 654--657.

\bibitem{Presic72} S.B. Pre\v si\' c, {\it Ein Satz \" Uber Reproduktive L\" osungen},
Publications de l'in\-sti\-tut mathematique, Nouvelle serie, tome 14 {\bf 28},
Beograd 1972, 133--136.

\bibitem{Bozic75} M. Bo\v zi\' c, {\it A Note On Reproductive Solutions}, Publications
de l'institut mathematique, Nouvelle serie, tome 19 {\bf 33}, Beograd 1975, 33--35.

\bibitem{Rudeanu78} S. Rudeanu, {\it On reproductive solutions of arbitrary equations},
Publications  de l'institut mathematique, Nouvelle serie, tome 24 {\bf 38}, Beograd 1978,
143--145.

\bibitem{Bankovic79} D. Bankovi\' c, {\it On general and reproductive solutions of
arbitrary equations}, Publications  de l'institut  mathematique, Nouvelle serie, tome 26
{\bf 40}, Beograd 1979, 31--33.

\bibitem{Keckic82} J.D. Ke\v cki\' c, {\it Reproductivity of some equations of analysis $I$},
Publications  de l'institut  mathematique, Nouvelle serie, tome 31 {\bf 45}, Beograd 1982,
73--81.

\bibitem{Keckic83} J.D. Ke\v cki\' c, {\it Reproductivity of some equations of analysis $II$},
Publications  de l'institut  mathematique, Nouvelle serie, tome 33 {\bf 47}, Beograd 1983,
109--118.

\bibitem{Ruski_Gruv_84_85} L.D. Dobryakov, {\it Commuting generalized inverse matrices},
Mathematical Notes, Volume 36, Number {\bf 1}, 500--504, 1985. (Translated from Matematicheskie
Zametki, Vol. 36, No. {\bf 52}, 17--23, 1984.)

\bibitem{Keckic85} J.D. Ke\v cki\' c, {\it Commutative weak generalized inverses of a square
matrix and some related matrix equations}, Publications  de l'institut  mathematique, Nouvelle
serie, tome 38 {\bf 52}, Beograd 1985, 39--44.

\bibitem{Bankovic95} D. Bankovi\' c, {\it All solutions of finite equations}, Discrete
Mathematics, Volume {\bf 137}, 1995, 1--6.

\bibitem{Bankovic96} D. Bankovi\' c, {\it Formulas of general solutions of Boolean equations},
Discrete Mathematics, Volume {\bf 152}, 1996, 25--32.

\bibitem{KeckicPresic97} J.D. Ke\v cki\' c  and  S.B. Pre\v si\' c, {\it Reproductivity - A
general approach to equations}, Facta Universitatis (Ni\v s), Ser. Math. Inform. {\bf 12} (1997),
157--184.

\bibitem{Keckic97} J.D. Ke\v cki\' c, {\it Some remarks on possible generalized inverses in
semigroups}, Publications  de l'institut  mathematique, Nouvelle serie, tome 61 {\bf 75}, Beograd
1997, 33--40.

\bibitem{Presic00} S.B. Pre\v si\' c, {\it A generalization of the notion of reproductivity},
Publications de  l'institut  mathematique, Nouvelle serie, tome 67 {\bf 81}, Beograd 2000,
76--84.

\bibitem{Ben-IsraelGreville03} A. Ben-Israel and T.N.E. Greville, {\it Generalized Inverses$:
$Theory and Applications}, Springer, 2003.

\bibitem{CampbellMeyer09} S.L. Campbell and C.D. Meyer, {\it Generalized Inverses of Linear
Transformations}, Society for Industrial and Applied Mathematics, 2009.

\bibitem{Bankovic11} D. Bankovi\' c, {\it General Solutions of System of Finite Equations},
Scientific Publications of the State University of Novi Pazar Ser. A: Appl. Math. Inform. and
Mech. vol. {\bf 3}, 2 (2011), 117--121.

\bibitem{MalesevicRadicic} B. Male\v sevi\' c and B. Radi\v ci\' c, {\it Non-reproductive and
reproductive solutions of some matrix equations},  Proceedings of the Conference {\em Mathematical
and Informational Technologies, MIT-2011}, V. Banja, Serbia,  2011, 246-251. (http:/$\!$/mit.rs/,
http:/$\!$/conf.nsc.ru/MIT-2011/)

\bibitem{RadicicMalesevic} B. Radi\v ci\' c and  B. Male\v sevi\' c, {\it Some considerations in
relation to the matrix equation $AXB=C$} (http:/$\!$/arxiv.org/abs/1108.2485)

\end{thebibliography}
\end{document}